\documentclass{article}
\usepackage[left=1in,right=1in,bottom=1in,top=1in]{geometry}

\usepackage{amsmath,amsfonts,amsthm, amssymb, hyperref}
\usepackage{booktabs}
\usepackage[capitalise]{cleveref}
\usepackage[dvipsnames]{xcolor}
\usepackage{graphicx}
\usepackage{caption}

\newcommand{\mS}{\ensuremath{\mathcal{S}}}

\newcommand{\bz}{\ensuremath{\mathbf{z}}}

\newcommand{\bx}{\ensuremath{\mathbf{x}}}
\newcommand{\by}{\ensuremath{\mathbf{y}}}

\newcommand{\bo}{\ensuremath{\mathbf{1}}}

\newcommand{\bzer}{\ensuremath{\mathbf{0}}}
\newcommand{\bone}{\ensuremath{\mathbf{1}}}

\newcommand{\balpha}{{\mbox{\ensuremath{\boldsymbol{\alpha}}}}}

\newcommand{\btheta}{{\mbox{\ensuremath{\boldsymbol{\theta}}}}}

\newcommand{\steps}{\mathcal{S}}

\usepackage{biblatex}
\DeclareCaptionFormat{myformat}{#1#2#3\hrulefill}
\newtheorem{theorem}{Theorem}
\newtheorem{lemma}{Lemma}
\newtheorem{definition}{Definition}
\newtheorem{conjecture}{Conjecture}

\title{Asymptotics of Weighted Reflectable Walks in $A_2$}

\author{Torin Greenwood\thanks{Department of Mathematics, North Dakota State University, Fargo, ND USA, \href{mailto:torin.greenwood@ndsu.edu}{torin.greenwood@ndsu.edu}}, Samuel Simon\thanks{Department of Mathematics, Simon Fraser University, Burnaby, BC, Canada \href{mailto:ssimon@sfu.ca}{ssimon@sfu.ca}}}

\begin{document}

\maketitle

\begin{abstract}
Lattice walks are used to model various physical phenomena.  In particular, walks within Weyl chambers connect directly to representation theory via the Littelmann path model.  We derive asymptotics for centrally weighted lattice walks within the Weyl chamber corresponding to $A_2$ by using tools from analytic combinatorics in several variables (ACSV).  We find universality classes depending on the weights of the walks, in line with prior results on the weighted Gouyou-Beauchamps model.  Along the way, we identify a type of singularity within a multivariate rational generating function that is not yet covered by the theory of ACSV.  We conjecture asymptotics for this type of singularity.  
\end{abstract}

\section{Lattice walks}

Lattice walks have a rich history both as a model of phenomena throughout math and science, and as a driving force for the development of new analytic techniques to extract asymptotics from general combinatorial problems.  For example, lattice walks have modeled melting phenomena in statistical mechanics \cite{Fisher:1984}, diffusion and Brownian motion \cite{BaFa:1987}, queueing systems \cite{FaIaMa:1999}, and Young diagrams \cite{Grabiner:2004, Konig:2005}.  Additionally, lattice walks have pushed forward the techniques of analytic combinatorics in several variables (ACSV), as the categorization of increasingly many families of lattice walks has continually stretched the limits of generating functions one can analyze \cite{BMMi:2010, CoMeMiRa:2017, MeWi:2019}.

This work continues the tradition, studying the asymptotics of reflectable weighted lattice walks within a Weyl chamber.  While this family of walks has direct connections to the Littelmann path model and representation theory \cite{Littelmann:1997}, the analysis here also reveals a type of singularity within a generating function previously unseen in applications.  Our main results include leading term asymptotics for weighted walks in the Tandem and Double Tandem models for almost all choices of \emph{central weightings}, as defined in \cref{sec:weights}.  Additionally, \cref{conj:TransverseBoundary} predicts asymptotics generally for generating functions in the new singularity regime we identified, based on merging the results on several related types of singularities.

A \emph{lattice model} in $d$ dimensions is defined by a finite stepset $\mathcal{S} \subset \mathbb{Z}^d$. A \emph{lattice walk} of length $n$, or \emph{lattice path} of length $n$, is a sequence $w = (w_1, w_2, \cdots, w_n)$ of steps $w_j \in \mathcal{S}$. 
After $m$ steps, the walk is at the point given by $\sum_{i=1}^m w_i$. We consider counting the number of walks restricted to the Weyl chamber $A_2$, defined in \cref{sec:Weyl} below.  As in \cref{fig:WeylChamber}, we will find that the walks we study could also be viewed as walks in the positive quarter plane, although the Weyl chamber interpretation allows us to use the generalized reflection principle \cite{GeZe:1992} to derive a generating function encoding the walks.  An alternate but related approach to the generalized reflection principle is the kernel method.

\subsection{Walks in restricted regions}
Dyck paths form a prototypical one-dimensional lattice path enumeration problem with a domain restriction: Dyck paths of length $2n$ start and end at $0$, take $2n$ steps from $\{1, -1\}$, and always remain at or above the point $0$.  
One way to show Dyck paths are enumerated by the Catalan numbers is to use the \emph{reflection principle}, where paths that do cross below $0$ are mapped bijectively to paths that are easier to count.

Natural extensions include counting walks in higher dimensions, with different stepsets, or in other restricted regions.  For one-dimensional walks, \cite{BaFl:2002} provides a generating function and asymptotic formula for restricted walks with general weighted stepsets, which assign a positive weight to each step.  Moving up one dimension, walks in the half plane $\mathbb{Z} \times \mathbb{N}$ can sometimes be reduced to pairs of one-dimensional weighted walks by treating the horizontal and vertical coordinates as independent walks.

When walks are otherwise restricted in multiple dimensions, the analysis is substantially more involved.  For walks in the quarter plane, 
\cite{FaIaMa:1999, BMMi:2010} provided a systematic approach for deriving a generating function for broad classes of stepsets, instead of developing ad-hoc methods for individual stepsets.  
Symmetry plays a major role in computing generating functions, which we explore in \cref{sec:Weyl}.  
Many additional works have contributed to the study of walks in the positive quadrant, including \cite{KuRa:2012, Bostan:2017, Krattenthaler:2015, Raschel:2012}.  

In \cite{MeMi:2016}, asymptotics are found for walks in the positive $d$-dimensional orthant with highly symmetric nontrivial stepsets. The authors of \cite{MeMi:2016} express the generating function as the diagonal of a multivariate rational function. They give asymptotics for such unweighted walks as a function of the stepset and number of dimensions.  By adding one degree of freedom, work in \cite{MeWi:2019} generalized these results and determined asymptotics for stepset models which are symmetric over all but one axis.

Considering other domain restrictions, \cite{DeWa:2015} gives asymptotic behavior of a multidimensional random walk in a general cone, including in Weyl chambers.  
In this work, Denisov and Wachtel provide a formula for counting the number of walks of length $n$ between two specified points in $d$-dimensional space. They show that such walks have asymptotics of the form $K \cdot \rho^n \cdot n^{-\ell-d/2}$. The value of $\ell$ is given as 
a function of the smallest eigenvalue of the Laplace-Beltrami operator, which can add a barrier to directly applying their theorem.  Furthermore, their approach can not give an explicit expression for the constant factor in the asymptotics.  The work of \cite{Duraj:2014} extends these results to additional cases, where a parameter of the weighted walks called the \emph{drift} no longer needs to be zero.  

Bostan, Raschel, and Salvy make explicit the results of Denisov and Wachtel in the case $d=2$ with the cone $\mathcal{R}=\mathbb{N}^2$. They determine asymptotic formulas for excursions for all 79 small step models in the quarter plane~\cite{BoRaSa:2014}. Bogosel et al. \cite{BoPeRaTr:2020} further extract results from Denisov and Wachtel and make explicit the cases $\mS \subset \{-1,0,1\}^3\setminus \{ \bzer \}$ with the cone $\mathcal{R}= \mathbb{N}^3$.  They study three-dimensional excursions by associating a spherical triangle to each model.

\subsection{Weighted walks} \label{sec:weights}

Many discrete models require non-uniform probabilities on the steps. Assigning weights to steps in a given model allows for a more detailed analysis of the asymptotic counting function. Through asymptotic analyses with weights, we can discover relations between aspects of the model and the asymptotic formula for the number of walks.

If each step $w_i$ in a walk $(w_1, w_2,\cdots, w_n)$ has associated weight $a_i$, define the weight of the walk as $\prod_{i=1}^n a_i$. 
If the weights are positive integers, we can interpret the weighted model as allowing colors or multisets of steps. Weights could represent probabilities when they sum to 1.  We restrict our attention to  \emph{central weights}, which are defined by the property that two walks having same length and endpoints must have the same weight.  Central weights can equivalently be defined by assigning a weight to each orthogonal axis.  We write $\balpha = (a, b)$ for two-dimensional central weights.

One goal of the work here is to provide an explicit connection between the weights of the steps in a walk and the subexponential asymptotic behavior of the walks.  This relationship is depicted in \cref{fig:TandemAsym}, illustrating the transitions between various subexponential regimes.  Because this description may be difficult to extract from the general results of \cite{DeWa:2015}, we prove the results directly.

Most similar to our results, a weighted version of the Gouyou-Beauchamps (GB) model was studied in \cite{CoMeMiRa:2017}, following the work in \cite{BMMi:2010, Bostan:2017} on the unweighted model.  Here, the stepset is $\mathcal{S} = \{(1, 0), (-1, 0), (-1, 1), (1, -1)\}$, and the coordinates of the steps are centrally weighted with $a, b > 0$.  In \cite[Theorem 1]{CoMeMiRa:2017}, the authors showed asymptotics are always of the form $\kappa V^{[n]}(i, j)\rho^n n^{-r}$ for constants $\rho$ and $r$ that depend on the weights $a$ and $b$, and a harmonic function $V^{[n]}(i, j)$ depending on the weights and parity of $n$.  In particular, the exponential growth $\rho$ is a continuous function of $a$ and $b$ across boundaries, while $r$ is not.  We observe this same behavior in \cref{thm:A2asym} below.

In \cite{CoMeMiRa:2017}, the authors also give a diagram of the subexponential regimes for the Tandem stepset without proof that matches our subexponential regimes in \cref{thm:A2asym} below, but we provide a complete description of the asymptotics with constant terms and additionally note a particularly challenging regime and a possible solution in \cref{conj:TransverseBoundary} below.

Finally, in \cite{MiSi:2020}, the second author and a collaborator found results for weighted walks in $A_1^d$ for arbitrary $d$.  Much of the work there provides a scaffold for the asymptotic analyses here, although the case of $A_2$ turns out to be more complicated for several reasons.  In particular, when using the asymptotic integral estimate described in \cref{thm:integral} below, the leading term for $A_2$ is sometimes difficult to find because many of the initial terms in the asymptotic expansion are zero.  The complexity in finding leading term asymptotics implies that it would be even more challenging to find full asymptotic expansions in these cases.

\subsection{Weyl chambers} \label{sec:Weyl}

Weyl groups allow us to generalize the notion of symmetric stepsets.  For a broad treatment of Weyl groups, see \cite{Humphreys:2012}.  Some core results on walks in Weyl chambers appear in \cite{GeZe:1992}.

\begin{figure} 
\begin{center}
    \includegraphics[width=.9\textwidth]{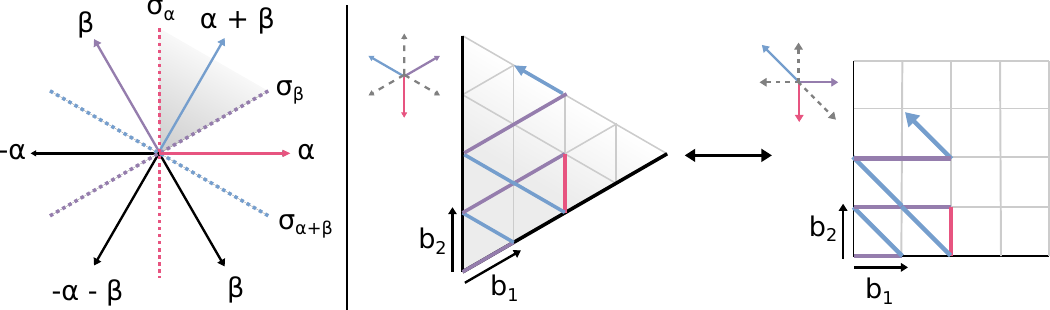}
\end{center}
\caption{The root system $\Phi = \{\pm \alpha, \pm \beta, \pm (\alpha + \beta)\} \subset \mathbb{R}^2$ appears on the left with a colored choice of positive roots.  The dotted lines illustrate the hyperplanes defining the Weyl group of reflections, $A_2$.  The fundamental Weyl chamber is shaded.  On the right, a walk in the chamber using the Tandem model stepset (colored), and the corresponding walk in the positive quadrant of $\mathbb{Z}^2$.  The Double Tandem stepset additionally includes the dashed lines.}
\label{fig:WeylChamber}
\end{figure}
\begin{definition}[Reduced Root System]
 For vectors $\bx, \by \in \mathbb{R}^d$, let $\sigma_{\bx}(\by)$ be the reflection of $\by$ through the hyperplane perpendicular to $\bx$.  A \emph{reduced root system} is a finite set of vectors $\Phi \subset \mathbb{R}^d$ such that for any $\bx, \by \in \Phi$: $\sigma_{\bx}(\by) \in \Phi$; $\by-\sigma_{\bx}(\by)$ is an integer multiple of $\bx$; and the only nontrivial scalar multiple of $\bx$ in $\Phi$ is $-\bx$.
\end{definition}

Root systems appear throughout math, especially in relation to Lie groups, and they capture important symmetry.  Given a root system $\Phi$, a special subset of \emph{positive roots} $\Phi^+$ can be chosen, where for each $\alpha \in \Phi$, exactly one of $\pm \alpha$ is in $\Phi^+$, and also if $\alpha, \beta \in \Phi^+$ and $\alpha + \beta \in \Phi$, then $\alpha + \beta \in \Phi^+$.  Then, as one more refinement, the elements of $\Phi^+$ which cannot be decomposed into sums of elements from $\Phi^+$ form a \emph{base} for $\Phi$.

The isometries defined by $\{\sigma_{\bx}: \bx \in \Phi\}$ form a group under composition, called a \emph{Weyl group}.  Additionally, the collection of hyperplanes associated to all of the isometries of the Weyl group partition $\mathbb{R}^d$ into regions called \emph{Weyl chambers}, as illustrated on the left in \cref{fig:WeylChamber}.  One of the chambers consists of points $v \in \mathbb{R}^d$ such that $\langle \gamma, v \rangle > 0$ for all $\gamma \in \Phi$, and this chamber is called the \emph{fundamental} or \emph{principal Weyl chamber}.  The root system, group of isometries, and principal Weyl chamber for $A_2$ are shown on the left in \cref{fig:WeylChamber}.  Finally, we define a reflectable stepset with respect to the Weyl group.

\begin{definition} \label{defn:reflectable}
Let $\mathcal{W}$ be a Weyl group acting on a real inner product space $V$ with a distinguished basis $\mathcal{B}=(\mathbf{b}_1,\dots, \mathbf{b}_d)$ and Weyl chamber $C$. We say that a nonempty set of vectors $\steps$ is a \emph{$(\mathcal{W},\mathcal{B})$-reflectable stepset} if for all $g\in \mathcal{W}$ and $s\in \steps$, we have $g(s)\in \steps$, and for all $s\in \steps$ and $1\leq i\leq d$, there is an integer $c_i$ such that the dot product $\langle s, \mathbf{b}_i \rangle \in \{-c_i,0,c_i\}$.
\end{definition}

For $A_2$, there are exactly two non-equivalent reflectable stepsets up to change of basis: the Tandem and Double Tandem stepsets illustrated in the middle in \cref{fig:WeylChamber}.  If the basis $\{b_1, b_2\}$ is chosen as unit vectors along the edge of the cone, then we can stretch the cone to a quadrant with axes corresponding to these basis vectors.  In this way, we can identify the walks within $A_2$ as walks in the positive quadrant of $\mathbb{Z}^2$.  Define $\mathcal{S}_T=\{(1,0)$,$(-1,1)$,$(0,-1)\}$ for the Tandem model, and $\mathcal{S}_{DT}= \{(1,0),(-1,1),(0,-1), (0,1),(1,-1),(-1,0)\}$ for the Double Tandem model.

Crucially, while $\mathcal{S}_T$ and $\mathcal{S}_{DT}$ do not appear symmetrical in the quarter plane, they are reflectable when considered within $A_2$.  Thus, the \emph{generalized reflection principle} can be used to analyze the number of walks within the chamber
\cite[Theorem 1]{GeZe:1992}.

Grabiner and Magyar gave exact results for walks in Weyl Chambers \cite{GrMa:1993}. Their formulas are for walks between two points staying within the designated chamber. They obtain these formulas using determinants. A number of their formulas include the hyperbolic Bessel function of the first kind of order $m$.

Grabiner later gave asymptotics for a number of Weyl Chambers including the region defined by $ x_1 \ge x_2 \ge \dots  \ge x_d$, which corresponds to the $d$-candidate ballot problem \cite{Grabiner:1999, Grabiner:2002, Grabiner:2004}. Here, the problem was interpreted as distributions of subtableaux in order to appeal to known formula for computing and manipulating Young tableaux. 

Krattenthaler \cite{Krattenthaler:2007} completed the asymptotic analysis for the number of random walks in a Weyl chamber and random walks on a circle, noting that computing the multiplicative constants remains a challenge.  Feierl extends this work by giving asymptotics for the zero drift reflectable walks in type $A$ Weyl chambers \cite{Feierl:2018}.  This work uses Taylor approximations and the saddle-point method to obtain asymptotics from known determinant formulas. Here, we derive results without using determinants, which leads to asymptotics of a simpler form.

\section{Results}

Here, we state our results for our asymptotic counts of weighted walks within $A_2$.  For the Tandem model, we recover the universality classes as found in \cite{CoMeMiRa:2017}, while also computing the asymptotic constants for almost all classes. We extend this to the Double Tandem model. 
In the exceptional cases when $a = 1, b < 1$ or $a < 1, b = 1$, we offer conjectured asymptotics and \cref{conj:TransverseBoundary}, a prediction for general asymptotics in such a regime.

\begin{theorem}\label{thm:A2asym}
    Let $\mathcal{R}= \mathbb{N}^2$ and let $\balpha = (a,b)$. For $\mathcal{S}_T=\{(1,0)$,$(-1,1)$,$(0,-1)\}$ (the Tandem model) and $\mathcal{S}_{DT}= \{(1,0),(-1,1),(0,-1), (0,1),(1,-1),(-1,0)\}$ (the Double Tandem model), the number of weighted walks of length $n$ which stay in $\mathcal{R}$ satisfies
    \begin{align*}
        q_{(a,b)}(n) \sim \gamma \rho^n n^{-r}
    \end{align*}
   where the exponential growth $\rho$ and subexponential growth $r$ for each of $\mathcal{S}_T$ and $\mathcal{S}_{DT}$ are given in \cref{fig:TandemAsym}, with the starred case conjectured.  The constant terms are given in \cite[Tables 5.3, 5.4]{Simon:2023}.
\end{theorem}

\begin{figure}[ht]
    \begin{center}
    $\left.
\begin{minipage}{.6\textwidth}
    \begin{center}
        \includegraphics[width=\textwidth]{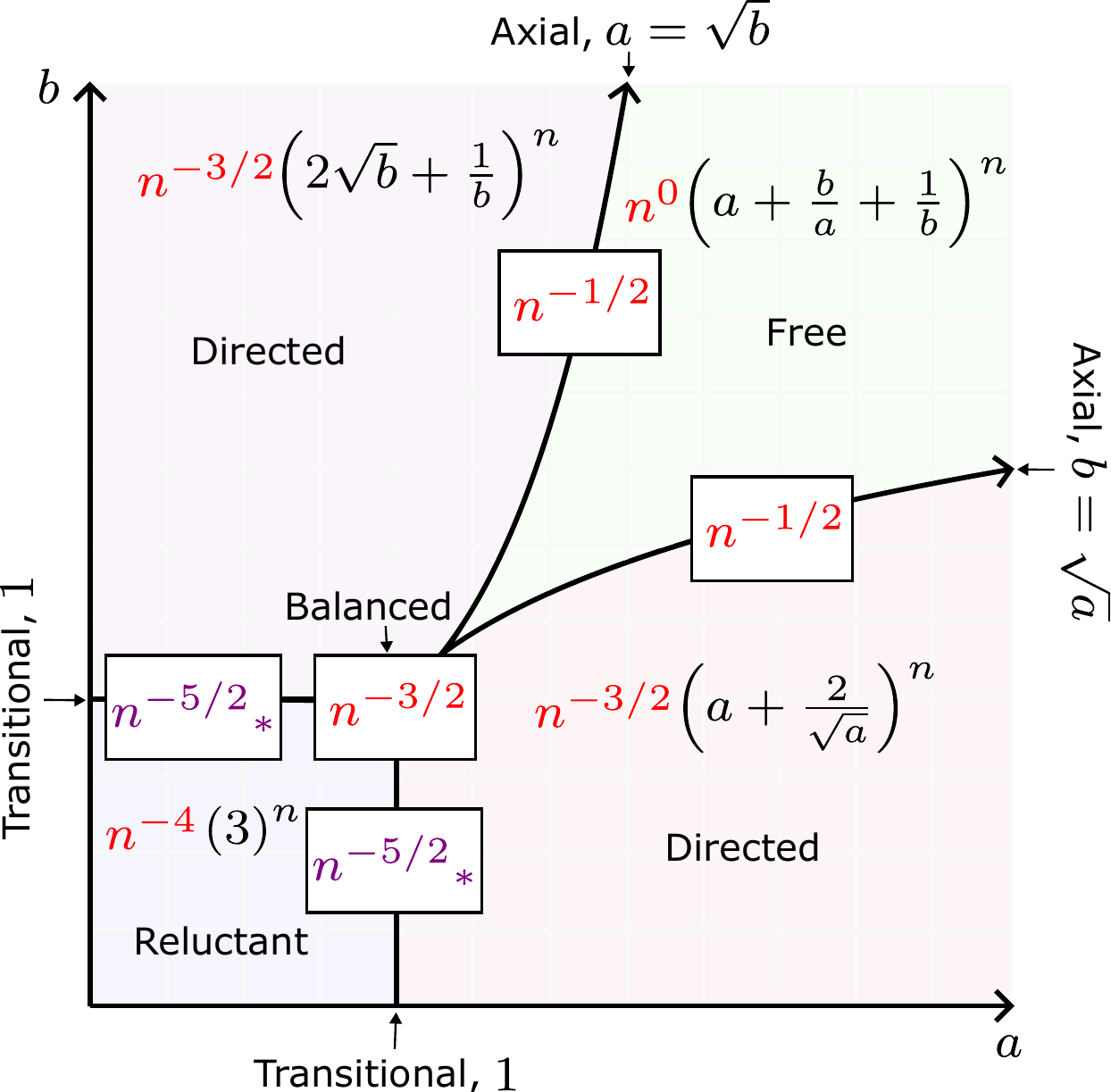}
    \end{center}
\end{minipage}
\ \right|$
\begin{minipage}{.36 \textwidth}
\begin{tabular}{c  c  c  c  c }
\midrule
Condition & DT $\rho$  \\
\midrule
$\scriptstyle 1<\sqrt{b}<a<b^2$ & $\scriptstyle \left( a+\frac{1}{a}+\frac{b}{a}+\frac{a}{b}+\frac{1}{b}+b \right)$
\\
\midrule
$\scriptstyle a=b^2 >1$  &  $\scriptstyle \left(\frac{a^2+2(a+1)\sqrt{a}+1}{a} \right)$ \\
$\scriptstyle b=a^2 >1$  &  $\scriptstyle  \left(\frac{b^2+2(b+1)\sqrt{b}+1}{b} \right)$
\\
\midrule
$\scriptstyle a>1, b<\sqrt{a}$ & $\scriptstyle \left(\frac{a^2+2(a+1)\sqrt{a}+1}{a} \right)$\\
$\scriptstyle b>1, a< \sqrt{b}$ & $\scriptstyle  \left(\frac{b^2+2(b+1)\sqrt{b}+1}{b} \right)$
\\
\midrule
$\scriptstyle b=a=1$  & 6
\\
\midrule
$\scriptstyle a<1,\ b=1$  & 6  \\
$\scriptstyle b<1,\ a=1$  & 6 
\\
\midrule 
$\scriptstyle a<1,\ b<1$  & 6
\\
\bottomrule
\end{tabular}
\end{minipage}
    \end{center}

    \caption{The Tandem and Double Tandem model have the same growth rate regimes with different exponential growth rates.  The regimes for the Tandem model are pictured on the left, with the subexponential growth (in red) and exponential growth (in black).
    The exponential growth is continuous across boundaries, and is unmarked on the boundaries.  
    On the right, the same regimes are listed with the corresponding exponential growth rates for the Double Tandem model.  Starred cases are conjectured.
    } \label{fig:TandemAsym}
\end{figure}

We verified the results given in \cref{thm:A2asym} numerically by computing $q_{(a,b)}(n)$ exactly for specific choices of $(a,b)$ in each regime and some large values of $n$ (see \cite[Table 5.4]{Simon:2023}).  In particular, we use the gfun Maple package provided by Salvy and Zimmermann~\cite{SaZi:1994}.

\section{Extracting asymptotics}
In many instances, analytic combinatorics in several variables (ACSV) provides a quick pipeline from a combinatorial description of a problem to asymptotics.  
Once a generating function is obtained, the singularities of the generating function can be classified.  Existing libraries of results (as in \cite{MePeWi:2024}) describe the asymptotics of the array for many of the most common types of singularities.

We represent a $d$-variate multivariate rational GF as $F(\mathbf{z}) := G(\mathbf{z})/H(\mathbf{z}) = \sum a_\mathbf{n} \mathbf{z}^\mathbf{n}$, where $\mathbf{z} = (z_1, \ldots, z_d)$ and $\mathbf{z}^\mathbf{n} = z_1^{n_1} \cdots z_d^{n_d}$.  The zero set $\mathcal{V} := \{\mathbf{z} : H(\mathbf{z}) = 0\}$ determines the singular variety of $F$.  We seek asymptotics for $[\mathbf{z}^{\mathbf{n}}]F(\mathbf{z})$ as $\mathbf{n} \to \infty$ in a prescribed \emph{direction} $\mathbf{\hat r} \in \mathbb{R}^d_{> 0}$, so that $\mathbf{n} \approx \mathbf{\hat r}n$ with $n \to \infty$.  In most combinatorial cases, finitely many \emph{critical points} determine the asymptotics of a generating function.  To find the critical points, consider representing the coefficients via the Cauchy integral formula,
\begin{equation} \label{eq:Cauchy}
    [\mathbf{z}^\mathbf{n}] F(\mathbf{z}) = \left(\frac{1}{2\pi i}\right)^d \int_T F(\mathbf{z}) \mathbf{z}^{-\mathbf{n} - 1} d\mathbf{z},
\end{equation}
where $T$ is a $d$-dimensional torus enclosing the origin but no singularities of $F$.  Heuristically, the critical points are determined by expanding $T$ until it reaches points on the singular variety closest to the origin that minimize the exponential growth $\mathbf{z}^{-\mathbf{n}}$ within the integrand.  

A critical point $\mathbf{p}$ is called \emph{smooth} if $\mathcal{V}$ is a smooth manifold in a neighborhood of $\mathbf{p}$.  This means that if $\mathcal{V}$ is $d$-dimensional, then in a neighborhood of $\mathbf{p}$ there is a smooth parameterization of $\mathcal{V}$ using only $d - 1$ variables. For rational generating functions, smoothness is easily checked using the implicit function theorem (see \cite[Lemma 7.6]{MePeWi:2024}).

For many lattice path enumeration problems, there are also \emph{transverse multiple points}, where $\mathcal{V}$ can locally be smoothly deformed into the intersection of perpendicular hyperplanes.  For rational GFs, these types of critical points satisfy systems of polynomial equations in terms of the denominator $H$ and its partial derivatives (see \cref{subsec:A2crit}).

Call a critical point $\mathbf{p} = (p_1, \ldots, p_d)$ \emph{minimal} when there are no other points $\mathbf{q} \in \mathcal{V}$ where $|q_i| \leq |p_i|$ for each coordinate with at least one inequality strict.  Smooth minimal critical points always contribute to asymptotics.  
However, for transverse minimal critical points, an additional technical condition must be met (\cref{def:contribcone}). 
A highlight of the analysis of weighted walks in $A_2$ is that there is a case where the technical condition is \emph{almost} met.  \cref{conj:TransverseBoundary} predicts this halves the contribution of the critical point to the asymptotics.

\section{Proof sketch}
We obtain the asymptotics in \cref{thm:A2asym} with the following steps:
\begin{enumerate}\setlength\itemsep{-.25em}
\item {\bf Encoding as a diagonal.} Using the symmetry group corresponding to the stepset, represent the generating function as a diagonal of a rational function.
\item {\bf Computing critical points.} Find the solutions to the critical point equations.
\item {\bf Finding contributing critical points.} Determine which critical points are contributing as a function of the weights.
\item {\bf Evaluating the Cauchy integral.} Simplify the Cauchy integral (\cref{eq:Cauchy}) to a Fourier-Laplace integral and then use existing results.
\end{enumerate}

\subsection{Encoding as a diagonal}
For both $\mS_T$ and $\mS_{DT}$, the reflection group is generated by the involutions $\Psi(x,y) = (y/x, y)$ and $\Phi(x,y) = (x, x/y)$.  
Using either the generalized reflection principle (as in \cite{MeMi:2016}) or evaluations of the unweighted generating functions in \cite[Examples 6.5.1 and 6.5.2]{MeWi:2019}, we find that weighted walks starting at the origin of length $n$ are encoded as coefficients of $x^ny^nt^n$ in the following functions.
 \begin{align}\label{eq:GF-T}
 F_T(x,y, t)&=\frac{G_T(x,y)}{H_T(x,y,t)} = \frac{(b^2x - ay^2)(bx^2 - a^2y)(xy-ab)}{(1-txy(\frac{a}{x} + \frac{bx}{ay}+\frac{y}{b}))a^3b^3(1-x)x(1-y)y}, \\
 F_{DT}(x,y,t)&=\frac{G_{DT}(x,y)}{H_{DT}(x,y,t)}\nonumber\\
 &= \frac{(b^2x - ay^2)(bx^2 - a^2y)(xy-ab)}{(1-txy(\frac{a}{x} + \frac{x}{a}+\frac{bx}{ay}+\frac{ay}{bx}+\frac{y}{b}+\frac{b}{y}))a^3b^3(1-x)x(1-y)y}. \label{eq:GF-DT}
 \end{align}

\subsection{Computing critical points}\label{subsec:A2crit}
First, we compute all possible critical points for all values of the weights $(a, b)$.  Then, in \cref{sec:ContribCrit}, we determine which critical points contribute to asymptotics.  We focus on the Tandem case here, as the Double Tandem case follows a similar analysis.

Weighted walks are encoded as the main diagonal of the functions in \cref{eq:GF-T}, so we search for critical points in the $\mathbf{1} = (1, 1, 1)$ direction.  By definition, smooth critical points satisfy $\{H = 0, xH_x = yH_y = tH_t\},$ where $H = H_T$. Next, to rule out non-smooth, non-transverse points, we verify that the factorization of $H_T$ given in \cref{eq:GF-T} is a \emph{transverse polynomial factorization} (as in \cite[Definition 9.3]{Melczer:2021}): define the \emph{inventory} $S(x, y) = ax + by/ax + 1/by$, and label the factors $H_0 = (1 - txyS(1/x, 1/y))$$, H_1 = (1-x), H_2 = (1-y)$.  At any point $\mathbf{w}$ where a factor $H_i(\mathbf{w}) = 0$, its gradient $\nabla H_i(\mathbf{w})$ is nonzero, and also at any point where the factors are simultaneously zero, their gradients are linearly independent.  (In fact, this applies broadly to GFs encoding other types of walks.)  This implies there are no non-smooth, non-transverse points.

To find the transverse multiple points, we must consider all 7 combinations of whether $H_0, H_1,$ and $H_2$ are zero, and
use \cite[Definition 9.7]{Melczer:2021} to compute the transverse critical points for each such \emph{stratum} individually.  Conveniently, the technical definition of transverse critical points simplifies greatly in these lattice walk cases where all but one of the factors are of the form $1 - x$ and $1 - y$.  For example, to compute the 
transverse critical points for $\mathcal{V}_{0, 1}$ (where $H_0, H_1 = 0$ and $H_2 \neq 0$), the equations simplify to using the smooth critical point equations on $H_0(1, y, t)$ to compute the $y$ and $t$ critical point coordinates.
Ultimately, we obtain the critical points in \cref{table:TandemCritPt} for each stratum.

\begin{table}[ht]
\begin{center}
\begin{tabular}{ c  c  c  c }
Stratum & $x$ & $y$ & $t$ \\
\hline
$\mathcal{V}_{0}$ & $a$ & $b$ & $ \frac{1}{3ab}$  \\
 &  $ae^{i (2\pi/3)}$ & $be^{i (4\pi/3)}$ & $ \frac{e^{i (4\pi/3)}}{3ab}$ \\
 &  $ae^{i (4\pi/3)}$ & $be^{i (2\pi/3)}$ & $ \frac{e^{i (2\pi/3)}}{3ab}$ \\
\hline 
$\mathcal{V}_{0,1}$ & $1$ & $b/\sqrt{a}$ & $ \frac{a^{5/2} - 2a}{b(a^3 - 4)}$ \\
& $1$ & $-b/\sqrt{a}$ & $-\frac{a^{5/2} + 2a}{b(a^3 - 4)}$ \\
\hline 
$\mathcal{V}_{0,2}$ & $a/\sqrt{b}$ & $1$ & $ \frac{2b^3 - b^{3/2}}{(4b^3 - 1)a}$ \\
& $-a/\sqrt{b}$ & $1$ & $\frac{2b^3+b^{3/2}}{(4b^3 - 1)a}$ \\
\hline 
$\mathcal{V}_{0,1,2}$ & $1$ & $1$ & $ \frac{1}{a + \frac{b}{a} + \frac{1}{b}}$
\end{tabular}
\caption {Critical points for each stratum, corresponding to every possible non-trivial choice of setting some of the factors $\{H_0, H_1, H_2\}$ to zero.}\label{table:TandemCritPt}
\end{center}
\end{table}

\subsection{Finding contributing critical points} \label{sec:ContribCrit}

We now refine to contributing critical points, starting by checking minimality.  The form of the generating function here is close enough to the Gouyou-Beauchamps generating function that we can reuse a result from \cite{CoMeMiRa:2017}.
\begin{lemma}[Lemma 3 of \cite{CoMeMiRa:2017}]\label{lemma:minimalclass}
    For the rational function $F(x,y,t)$ described by \eqref{eq:GF-T}, when $G$ and $H$ are coprime the point $(x,y,t)\in \mathcal{V}$ is minimal if and only if
\begin{align*}
\vert x \vert \le 1, \quad \quad \vert y \vert \le 1, \quad \quad \vert t \vert \le 
\frac{1}{\vert x y \vert S(\vert \frac{1}{x} \vert, \vert \frac{1}{y} \vert ) },
\end{align*}
where the three strict inequalities do not occur simultaneously. 
\end{lemma}

Next, we filter to minimal points minimizing the height function $|xyt|^{-1}$.

\begin{lemma}\label{prop:minimala2}
    For each value of $a,b$, the unique positive minimal point that minimizes the height function $|xyt|^{-1}$ is given in \cref{table:TandemCrit}.
\end{lemma}

\begin{table}[h!]
\begin{center}
\def\arraystretch{1.6}
\begin{tabular}{|c | c | c | c|}
\hline
CP & Conditions on weights & Positive minimal critical point & Exponential growth  \\
\hline
 1&$1<\sqrt{b}<a<b^2$ & $ x=1, y=1, t=\frac{1}{b+a/b+1/a} $ & $a + \frac{b}{a} + \frac{1}{b}$\\
\hline
 2&$a>1, b\leq\sqrt{a}$  & $ x=1, y= \frac{b}{\sqrt{a}}, t= \frac{a^{5/2} - 2a}{b(a^3 - 4)}  $ & $a + \frac{2}{\sqrt{a}}$\\
 \hline
 3&$b>1, a\leq \sqrt{b}$ & $ x= \frac{a}{\sqrt{b}}, y=1,  t= \frac{2b^3 - b^{3/2}}{4ab^3 - a}$ & $2\sqrt{b} + \frac{1}{b}$\\
\hline
 4&$a\leq1, b\leq1$  & $x=a,y=b, t=\frac{1}{3ab}$ & $3$\\
 \hline
\end{tabular}
\\\ \\
\caption {Positive minimal critical points for choices of the weights $a$ and $b$.}\label{table:TandemCrit}
\end{center}
\end{table}

\begin{proof}[Proof (sketch)]
Minimizing the height $|xyt|^{-1}$ is equivalent to minimizing $ \vert S(1/x, 1/y) \vert$,  which we can accomplish using calculus.  Both here and in arbitrary dimension, the contributing critical points display a non-obvious boolean lattice structure in the following sense.  Any given critical point is minimal when each of its coordinates (except the $t$ coordinate) is at most $1$.  If there is a minimal critical point with coordinate $x_i \neq 1$ and another minimal critical point with coordinate $x_j \neq 1$, then there must also be a minimal critical point where $x_i \neq 1$ and $x_j \neq 1$.  This greatly simplifies the problem of finding contributing critical points because it is easy to show that the more coordinates are equal to one in a critical point, the less the corresponding exponential growth is.  Then, from the boolean structure, there is never a need to compare the contributions of two different critical points with the same number of non-one coordinates.
\end{proof}

When there are only finitely many smooth minimal critical points, we can use existing results to compute asymptotics, but we need an additional definition and criterion in the presence of transverse multiple points.

\begin{definition}[Definition 9.8 of \cite{Melczer:2021}]\label{def:contribcone}
Let $H(\bz) =H_1(\bz) \cdots H_m(\bz)$ be a square-free factorization of $H$. Fix $K = \{k_1, \cdots, k_q\} \subseteq \{1, \ldots, m\}$, and let ${\bf w} \in \mathbb{C}^d$ be a solution to the critical point equations for the stratum where $H_i = 0$ if and only if $i \in K$.
For each $1 \le j \le q$ let $b_j \in \{1, \cdots, d\}$ be an index such that the partial derivative $(\partial H_{k_j} / \partial z_{b_j})({\bf w}) \ne 0$. The vector 
\begin{align*}
    {\bf v}_{k_j} = \frac{ (\nabla_{\log} H_{k_j})({\bf w})}{w_{b_j}(\partial H_j / \partial z_{b_j})({\bf w})} = 
    \left( 
    \frac{ w_{1}(\partial H_{k_j} / \partial z_{1})({\bf w})}{w_{b_j}(\partial H_j / \partial z_{b_j})({\bf w})}, \cdots, 
    \frac{ w_{d}(\partial H_{k_j} / \partial z_{d})({\bf w})}{w_{b_j}(\partial H_j / \partial z_{b_j})({\bf w})} 
    \right)
\end{align*}
has real coordinates. The \emph{normal cone} of $H$ at $\bf w$ is the set
\begin{equation} \label{eq:NormalCone}
    N({\bf w}) = \left\{ \sum_{j=1}^q a_j {\bf v}_{k_j} : a_j>0 \right\} \subset \mathbb{R}^d.
\end{equation}
The point $\bf w$ is called a \emph{contributing point} if  $\bf w$ is minimal, $\bf w$ minimizes $\vert {\bf z} \vert^{-1}$ among all minimal points, and ${\bo} \in N({\bf w})$.
\end{definition}

In some regimes below, it turns out that $\bone$ is in the boundary of $N({\bf w})$ (i.e. some $a_j$ must be 0).  Although $\mathbf{w}$ then does not meet the requirements to be a contributing point, it may still determine asymptotics.  The following lemma applies to both the critical points in \cref{table:TandemCrit}, and also to the critical points more generally for reflectable walks in $A_2^d$.  

\begin{lemma} \label{lem:ContribCP}
Let $\bf{w}$ be a  minimal critical point. If $\bf{w}$ has a coordinate of $1$ and $\mathbf{w}$ satisfies the smooth critical point equations for $H_0 = 1 - txyS(\frac{a}{x}, \frac{b}{y})$ in the direction $\bone$, then $\bone$ is on the boundary of the normal cone $N({\bf w})$ (see \cref{def:contribcone}).  Otherwise, $\bone$ is on interior of $N(\mathbf{w})$.
\end{lemma}

\begin{proof}[Proof (sketch)]
For $H_T$ factored as in \cref{eq:GF-T}, we can compute $\nabla_{\log}H_i(\mathbf{w})$ for $i = 0, 1, 2$ explicitly.  When $i = 1$ or $2$, the logarithmic gradient is a basis vector.  For $i = 0$,
\begin{align*}
    \nabla_{\log}H_0(\mathbf{w}) =\left(-1-xyt\left(-\frac{a}{x}+\frac{bx}{ay}\right), -1-xyt\left(-\frac{bx}{ya}+\frac{y}{b}\right),-1
    \right).
\end{align*}
In cases where $H_1(\mathbf{w}) = 0$, $\bone$ is in the interior cone if and only if $-a/x + bx/ay < 0$, and it is on the boundary if  $-a/x + bx/ay = 0$.  A similar statement can be made for $H_2$.  It is then a matter of algebra to show that equality occurs exactly when the critical point equations for $H_0 = 1 - txyS(1/x, 1/y)$ are met. \end{proof}

From \cref{lem:ContribCP}, we find that the critical points from \cref{table:TandemCrit} always contribute except perhaps when $a = 1$ or $b = 1$.  In these exceptional cases, we note that \cite[Theorem 10.65]{MePeWi:2024} indicates that when the numerator of a GF is nonzero at the critical point, the direction being on a facet of $N(\mathbf{w})$ cuts the asymptotic contribution in half.  Here, the numerator is zero, but we conjecture the idea is still true regardless.

\begin{conjecture} \label{conj:TransverseBoundary}
    When a direction $\mathbf{r}$ is on a facet of the normal cone $N(\mathbf{w})$ defined by a minimal transverse critical point $\mathbf{w}$, then $\mathbf{w}$ contributes half as much to the asymptotics as when $\mathbf{r}$ is in the interior.
\end{conjecture}
As with all of the other regimes for the Tandem and Double Tandem model, we have verified this conjecture numerically when $a = 1$ or $b = 1$, and our conjectured subexponential growth aligns with Figure 7 of \cite{CoMeMiRa:2017}.  In particular, we looked at the weights $a = 1/8$ and $b = 1$ and found that for walks of length $2000$, the error between the asymptotic estimate and the exact number of walks is less than $1\%$.  For weights $a = 1$ and $b = 1/4$ and walks of length $2000$, the error was approximately $1.2\%$.

Note that this situation does not occur in the analysis of the Gouyou-Beauchamps walks in \cite{CoMeMiRa:2017}.  This is because in the transitional cases for the Gouyou-Beauchamps walks, the corresponding generating function has a factor of $1 - y$ in the numerator and denominator that cancels and makes these cases among the easier cases to analyze.  This is notable in particular because the factor of $1 - y$ in the numerator is independent of the weights in this regime.  Although we too find cancellation of factors in the numerator and denominator for some regimes (see \cref{sec:Axial} below), there is no cancellation in the transitional cases for the Tandem or Double Tandem models, and indeed there is no factor in the numerator that is independent of the weights.

\subsection{Evaluating the Cauchy integral}

The final step is to set up the integral to compute asymptotics. Note that the textbook \cite{MePeWi:2024} includes results for transverse critical points that could be applied directly at this point, but for a more complete and elementary viewpoint, we include a residue approach.  Beginning with the Cauchy integral equation (\cref{eq:Cauchy}), we expand the torus $T$ until it nears the minimal critical points in \cref{table:TandemCrit}.  When different minimal points from \cref{table:TandemCrit} end up being equal at certain weight values, the analysis differs in these cases because it causes cancellation between factors of $G$ and $H$.  Ultimately, we are left with the 9 cases as described in \cref{fig:TandemAsym}.  We outline here an overview of the process of extracting asymptotics.  The details for each of the 9 cases can be found in \cite{Simon:2023}, with an example in \cref{sec:Axial} below.  We also include {\tt SageMath} code at the following URL illustrating how to compute asymptotics in each of these cases.
\begin{center}
    \url{https://cocalc.com/TorinGreenwood/AofA-A2Walks/A2TandemDoubleTandem}
\end{center}

The overall goal is to simplify the integral until it is a Fourier-Laplace type integral where the following result applies:
\begin{theorem}[Theorem 7.7.3 of \cite{Hormander:2015}; Lemma 5.16 of \cite{MePeWi:2024}]\label{thm:integral} 
Suppose that the functions $A(\btheta)$ and $\phi(\btheta)$ in $r$ variables are smooth in a neighborhood $\mathcal{N}$ of the origin and that the gradient $\nabla \phi(\bzer)=\bzer$; the Hessian $\mathcal{H}$ of $\phi$ at $\bf{0}$ is non-singular; $\phi(\bf{0})=0$; and the real part of $\phi(\btheta)$ is non-negative on $\mathcal{N}$.  Then for each $M>0$ there are complex constants $C_0, \dots, C_M$ such that 
\begin{equation} \label{eq:IntApprox}
\int_{\mathcal{N}} A (\btheta) e^{-n \phi(\btheta)} d \btheta  = \left( \frac{ 2 \pi }{n} \right)^{r/2} \det{\left( \mathcal{H} \right)}^{-1/2} \cdot \sum_{j=0}^M C_j n^{-j} + O(n^{-M-1}).
\end{equation}
The constants $C_j$ are given by the formula: 
\begin{equation}\label{eq:Ck}
C_j =  (-1)^j  \sum_{\ell \le 2j} \frac{ \mathcal{D}^{\ell+j} ( A \underline{\phi}^\ell)\bf(0)}{2^{\ell+j} \ell! (\ell+j)!},
\quad\text{ with }
\quad\underline{\phi} := \phi - \langle \btheta, \mathcal{H} \btheta \rangle
\end{equation}
where $\mathcal{D}$ is the  differential operator
$
\mathcal{D} := \sum_{u,v} (\mathcal{H}^{-1})_{u,v} \, \frac{\partial}{\partial\theta_u} \frac{\partial}{\partial\theta_v}.
$
\end{theorem}

The computational obstacle in using \cref{thm:integral} is determining the first $j$ for which $C_j$ is nonzero, as this gives the subexponential growth. If $G$ vanishes to order $k$ at the critical point, then $C_j = 0$ for $0 \leq j \leq \lceil k/2 \rceil - 1$.
Whenever the critical point is not smooth, we first take residues to reduce the number of variables in the integral and also make the singular variety smooth.  Because the non-smoothness comes from factors of the form $(1 - x)$ or $(1 - y)$, it is typically straightforward to compute residues.

For example, when a critical point has $x$ coordinate equal to $1$, we can compare the value of the integral over the circle $\vert x \vert = 1 - \epsilon$ to the integral at $\vert x \vert=1+ \epsilon$ and add a term which has smaller exponential growth, so it does not contribute to the dominant asymptotics. Then we compute the difference of the two integrals using the residue theorem, which corresponds to removing the factor of $(1-x)$ in the denominator and evaluating the remaining function at $x=1$. After applying the residue, we check to see if factors between $G$ and $H$ now cancel, which can impact the order to which $G$ vanishes. Then, we do a change of variables to set the integral to use \cref{thm:integral}. Lastly, we compute the $C_j$ to obtain the asymptotics, which is completed using code.  We incorporate portions of the code available in the online supplement to the textbook, \cite{Melczer:2021}.

\subsection{Example analysis: axial regime} \label{sec:Axial}

Here, we compute the asymptotics in the case where $a = b^2 > 1$.  Equivalently, by expanding the generating function in \cref{eq:GF-T} as a geometric series in $t$, we aim for an asymptotic expression for the following:
\begin{align*}
q_{(a,b)}(n) :=[x^0][y^0]\left( \frac{a(x - y^2)(a^{1/2}x^2 - a^2y)(a^{3/2} - xy)}{a^{9/2}(x - 1)x(y - 1)y} \left(\frac{a}{x}+\frac{x}{a^{1/2}y}+ \frac{y}{a^{1/2}}   \right)^n \right).
 \end{align*}

The critical point that is contributing is $\left(1,  \frac{b}{\sqrt{a}} \right) = (1,1)$.
However, we calculate that the direction $(1,1)$ is not in the normal cone at this point, and is instead on the boundary.
To get around this, we  take the term $(x-y^2)$ in the numerator and express it at $(x-1)-(y^2-1)$. Since coefficient extraction is linear, we have the following
 \begin{align*}
q_{(a,b)}(n) &=[x^0][y^0]\left( \frac{a(a^{1/2}x^2 - a^2y)(a^{3/2} - xy)}{a^{9/2}x(y - 1)y} \left(\frac{a}{x}+\frac{x}{a^{1/2}y}+ \frac{y}{a^{1/2}}   \right)^n \right)
 \\
 &+
[x^0][y^0]\left( \frac{a(y+1)(a^{1/2}x^2 - a^2y)(a^{3/2} - xy)}{a^{9/2}(x - 1)xy} \left(\frac{a}{x}+\frac{x}{a^{1/2}y}+ \frac{y}{a^{1/2}}   \right)^n \right).
 \end{align*}
 
 The first function has critical point at $(a,1)$. The second function has critical point $(1,1)$. Thus, the first function does not contribute to the asymptotics.  The cancellation of factors here is similar to \cite{CoMeMiRa:2017}.
 
 In order to obtain asymptotics from the second function, we start by taking a residue at $x=1$.
 The next step is to do a change of variables to make it of Fourier-Laplace type so we can use \cref{thm:integral}. 
We apply the change of variables $y = e^{i \theta}, dy = i e^{i \theta} d \theta$, so the region of integration is over $[-\pi/2, 3\pi/2)$. 
With this transformation the integral becomes 
\[
\int\displaylimits_{[-\pi/2, 3\pi/2)} A(\theta) e^{-n \phi(\theta)} d \theta,
\]
where
\[
A(\theta) = \frac{(a^{2}e^{i \theta}-a^{1/2})(a^{3/2}-e^{i \theta})(e^{i \theta}+1)e^{-i \theta}}{a^{7/2}}
\]
and
\[\phi(\theta) = \log{\left( \frac{a+\frac{2}{\sqrt{a}}}{a^{-1/2}e^{2 i \theta}+ae^{i \theta}+a^{-1/2}} \right)}.\]
Applying \cref{thm:integral} gives the formula
\begin{align*}
q_{(a,b)}(n) \sim (a+2a^{-1/2})^n \cdot n^{-1/2} \cdot
\frac{(a^3 - 2a^{3/2} + 1)\sqrt{a^{3/2} + 2}}{\sqrt{\pi}a^3}.
\end{align*}
For the Double Tandem stepset we compute
\begin{multline*}
q_{(a,b)}(n) \sim \left(\frac{a^2+2(a+1)\sqrt{a}+1}{a} \right)^n \cdot n^{-1/2}\\
\cdot \frac{(a^{7/2} - 2a^{2} + \sqrt{a})\sqrt{2a^2+(a^2+1)\sqrt{a}+2a}}{\sqrt{\pi}a^4\sqrt{a+1}}.
\end{multline*}

\section{Next steps}

The results here merely scratch the surface of possible questions about walks within Weyl chambers.  An obvious next step would be to analyze the $d$-dimensional Tandem and Double Tandem stepsets.  For example, the $d$-dimensional Tandem stepset has steps given by
\[
\mS_{T_d}= \{ e_i-e_{i-1}: 2 \le i \le d\} \cup \{e_1\} \cup \{-e_d\}\]
where $e_i$ is the $i$th elementary basis vector with a one in the $i$th coordinate and zeroes elsewhere.  The first steps in computing the asymptotics are not the main obstructions. We can express the generating function for these walks as the diagonal of a rational function, and solve the critical point equations in $d$ dimensions.  We additionally find a similar structure to the contributing critical points as in the $2$-dimensional case.  However, there are more cases where \cref{conj:TransverseBoundary} may apply and computing constants becomes increasingly difficult.

These difficulties appear largely because applying \cref{thm:integral} 
involves solving for the first nonzero $C_j$ in \cref{eq:IntApprox}.  This is in contrast to existing results for $A_1^d$, where the functional form of the group sum in the $A_1^d$ case allowed the authors in \cite{MiSi:2020} work through the calculations in general. In particular, the first nonzero $C_j$ was always the first term where there are nonzero derivatives of order $2j$. For $A_d$ it is straightforward to determine the degree to which the function vanishes at a critical point, but this is not sufficient. For $A_2$ when $a<1$ and $b<1$, the function vanishes to degree three but the constant $C_2$ is still zero at the critical point. It is possible that there are aspects of the governing function, coming from the Weyl denominator, that must be exploited in order to give a general statement. Even for $A_3$, computations can include taking 90 different mixed partial derivatives of order 24. Certainly, there are simplifications that can be made to obtain this, but it presents a barrier to quickly getting results in higher dimensions to find a pattern. 

While current work has focused on the Weyl chambers of $A_1^d$ and $A_d$, there are other families of interest. In particular, there may still be room to use the approach here to derive explicit asymptotic results for weighted reflectable walks for the family of Weyl groups $B_d$ for $d > 2$. In \cite{Feierl:2014} Feierl counted weighted walks in $B_d$ using determinants, while the case of weighted reflectable walks in $B_2$ has been covered in \cite{CoMeMiRa:2017}.

More generally, one goal is to have results for walks in the product of any Weyl chambers. This would be the culmination of multiple projects, as there are not general results for all Weyl chambers. This is a plausible project as the product of the chambers should decompose in the same sense as the reflectable walks.

\printbibliography 

@article{BaFa:1987,
  title={Analysis of models reducible to a class of diffusion processes in the positive quarter plane},
  author={Baccelli, F. and Fayolle, G.},
  journal={SIAM Journal on Applied Mathematics},
  volume={47},
  number={6},
  pages={1367--1385},
  year={1987},
  publisher={SIAM},
  doi = {10.1137/0147090}
}

@article{BaFl:2002,
  title={Basic analytic combinatorics of directed lattice paths},
  author={Banderier, C. and Flajolet, P.},
  journal={Theoretical Computer Science},
  volume={281},
  number={1-2},
  pages={37--80},
  year={2002},
  publisher={Elsevier},
  doi = {10.1016/S0304-3975(02)00007-5}
}

@article{BMMi:2010,
  	title={Walks with small steps in the quarter plane},
  	author={Bousquet-M{\'e}lou, M. and Mishna, M.},
  	journal={Contemporary Mathematics},
  	volume={520},
  	pages={1--40},
  	year={2010},
    doi = {10.1090/conm/520}
}

@article{BoPeRaTr:2020,
  title={3D positive lattice walks and spherical triangles},
  author={Bogosel, B. and Perrollaz, V. and Raschel, K. and Trotignon, A.},
  journal={Journal of Combinatorial Theory, Series A},
  volume={172},
  pages={105189},
  year={2020},
  publisher={Elsevier},
  doi = {10.1016/j.jcta.2019.105189}
}

@article{Bostan:2017,
  title={Hypergeometric expressions for generating functions of walks with small steps in the quarter plane},
  author={Bostan, A. and Chyzak, F. and van Hoeij, M. and Kauers, M. and Pech, L.},
  journal={European Journal of Combinatorics},
  volume={61},
  pages={242--275},
  year={2017},
  publisher={Elsevier},
  doi = {10.1016/j.ejc.2016.10.010}
}

@article{BoRaSa:2014,
	Author = {Bostan, A. and Raschel, K. and Salvy, B.},
	Journal = {Journal of Combinatorial Theory, Series A},
	Number = {0},
	Pages = {45--63},
	Title = {Non-{D}-finite excursions in the quarter plane},
	Volume = {121},
    Year = {2014},
    doi = {10.1016/j.jcta.2013.09.005}
}

@article{CoMeMiRa:2017,
  title={Weighted lattice walks and universality classes},
  author={Courtiel, J. and Melczer, S. and Mishna, M. and Raschel, K.},
  journal={Journal of Combinatorial Theory, Series A},
  volume={152},
  pages={255--302},
  year={2017},
  publisher={Elsevier},
  doi = {10.1016/j.jcta.2017.06.008}
}

@article{DeWa:2015,
  title={Random walks in cones},
  author={Denisov, D. and Wachtel, V.},
  journal={The Annals of Probability},
  volume={43},
  number={3},
  pages={992--1044},
  year={2015},
  publisher={Institute of Mathematical Statistics},
  URL = {http://www.jstor.org/stable/24519214}
}

@article{Duraj:2014,
title = {Random walks in cones: The case of nonzero drift},
journal = {Stochastic Processes and their Applications},
volume = {124},
number = {4},
pages = {1503-1518},
year = {2014},
issn = {0304-4149},
doi = {10.1016/j.spa.2013.12.003},
author = {Jetlir Duraj}
}

@book{FaIaMa:1999,
	Address = {Berlin},
	Author = {G. Fayolle and R. Iasnogorodski and V. Malyshev},
	Pages = {xvi+156},
	Publisher = {Springer-Verlag},
	Series = {Applications of Mathematics (New York)},
	Title = {Random walks in the quarter-plane},
	Volume = {40},
    Year = {1999},
    doi = {10.1007/978-3-319-50930-3}
}

@article{Feierl:2014,
	Author = {Feierl, T.},
	Doi = {10.1002/rsa.20467},
	Fjournal = {Random Structures \& Algorithms},
	Issn = {1042-9832},
	Journal = {Random Structures Algorithms},
	Mrclass = {05A16},
	Mrnumber = {3245292},
	Mrreviewer = {Charlotte Alix Brennan},
	Number = {2},
	Pages = {261--305},
	Title = {Asymptotics for the number of walks in a {W}eyl chamber of type {$B$}},
	Volume = {45},
	Year = {2014}}

@article{Feierl:2018,
  title={Asymptotics for the number of zero drift reflectable walks in a {W}eyl chamber of type {A}},
  author={Feierl, T.},
  journal={Preprint arXiv:1806.05998},
  year={2018},
  doi={10.48550/arXiv.1806.05998}
}

@article{Fisher:1984,
  title={Walks, walls, wetting, and melting},
  author={Fisher, M. E.},
  journal={Journal of Statistical Physics},
  volume={34},
  number={5},
  pages={667--729},
  year={1984},
  publisher={Springer},
  doi={10.1007/BF01009436}
}

@article {GeZe:1992,
    AUTHOR = {Gessel, I. M. and Zeilberger, D.},
     TITLE = {Random walks in a {W}eyl chamber},
   JOURNAL = {Proceedings of the American Mathematical Society},
  FJOURNAL = {Proceedings of the American Mathematical Society},
    VOLUME = {115},
      YEAR = {1992},
    NUMBER = {1},
     PAGES = {27--31},
   MRCLASS = {05A15},
  MRNUMBER = {1092920},
MRREVIEWER = {Sri Gopal Mohanty},
doi = {10.1090/S0002-9939-1992-1092920-8}
}

@article{Grabiner:1999,
  title={Brownian motion in a {W}eyl chamber, non-colliding particles, and random matrices},
  author={Grabiner, D. J.},
  journal={Annales de l'Institut Henri Poincar{\'e} Probabilit{\'e}s et Statistiques},
  volume={35},
  number={2},
  pages={177--204},
  year={1999},
  doi={10.1016/S0246-0203(99)80010-7}
}

@article{Grabiner:2002,
  title={Random walk in an alcove of an affine {W}eyl group, and non-colliding random walks on an interval},
  author={Grabiner, D. J.},
  journal={Journal of Combinatorial Theory, Series A},
  volume={97},
  number={2},
  pages={285--306},
  year={2002},
  publisher={Elsevier},
  doi={10.1006/jcta.2001.3216}
}

@article{Grabiner:2004,
  title={Asymptotics for the distributions of subtableaux in {Y}oung and up-down tableaux},
  author={Grabiner, D. J.},
  journal={Electronic Journal of Combinatorics},
  volume={11(2)},
  pages={R29},
  year={2006},
  doi={10.37236/1886}
}

@article{GrMa:1993,
	Author = {D. J. Grabiner and P. Magyar},
	Journal = {J. Algebraic Combin.},
	Number = {3},
	Pages = {239--260},
	Title = {Random walks in {W}eyl chambers and the decomposition of tensor powers},
	Volume = {2},
    Year = {1993},
    doi = {10.1023/A:1022499531492}
}

@book{Hormander:2015,
  title={The analysis of linear partial differential operators I: Distribution theory and Fourier analysis},
  author={H{\"o}rmander, L.},
  year={2015},
  publisher={Springer},
  doi={10.1007/978-3-642-61497-2}
}

@book{Humphreys:2012,
  title={Introduction to Lie algebras and representation theory},
  author={Humphreys, J. E.},
  volume={9},
  year={2012},
  publisher={Springer Science \& Business Media},
  doi={10.1007/978-1-4612-6398-2}
}

@article{Konig:2005,
  title={Orthogonal polynomial ensembles in probability theory},
  author={K{\"o}nig, W.},
  journal={Probability Surveys},
  volume={2},
  pages={385--447},
  year={2005},
  doi={10.1214/154957805100000177}
}

@article{Krattenthaler:2007,
  title={Asymptotics for random walks in alcoves of affine {W}eyl groups},
  author={Krattenthaler, C.},
  journal={S{\'e}minaire Lotharingien de Combinatoire},
  volume={52},
  pages={B52i},
  year={2007},
  URL={https://www.emis.de/journals/SLC/wpapers/s52kratt.html}
}

@article{Krattenthaler:2015,
  title={Lattice path enumeration},
  author={Krattenthaler, C.},
  journal={Handbook of enumerative combinatorics},
  pages={589--678},
  year={2015},
  publisher={CRC Press Boca Raton, FL},
  doi={10.1201/b18255}
}

@article{KuRa:2012,
	Author = {I. Kurkova and K. Raschel},
	Journal = {Publ. Math. Inst. Hautes \'Etudes Sci.},
	Pages = {69--114},
	Title = {On the functions counting walks with small steps in the quarter plane},
	Volume = {116},
    Year = {2012},
    doi = {10.1007/s10240-012-0045-7}
}

@article{Littelmann:1997,
  title={Characters of representations and paths in $\mathfrak{H}_\mathbb{R}^*$ },
  author={Littelmann, P.},
  journal={Representation theory and automorphic forms (Edinburgh, 1996)},
  volume={61},
  pages={29--49},
  year={1997},
  doi={10.1090/pspum/061}
}

@book{Melczer:2021,
	author = {Melczer, Stephen},
	isbn = {978-3030670795},
	publisher = {Springer Nature},
	title = {An Invitation to Analytic Combinatorics: From One to Several Variables},
    year = {2021},
    doi={10.1007/978-3-030-67080-1}
}

@article {MeMi:2016,
    AUTHOR = {Melczer, S. and Mishna, M.},
     TITLE = {Asymptotic lattice path enumeration using diagonals},
   JOURNAL = {Algorithmica},
  FJOURNAL = {Algorithmica. An International Journal in Computer Science},
    VOLUME = {75},
      YEAR = {2016},
    NUMBER = {4},
     PAGES = {782--811},
   MRCLASS = {05A15 (05A16)},
  MRNUMBER = {3509390},
MRREVIEWER = {Christoph Bernhard Koutschan},
       DOI = {10.1007/s00453-015-0063-1},
}

@book{MePeWi:2024, place={Cambridge}, edition={2}, series={Cambridge Studies in Advanced Mathematics}, title={Analytic Combinatorics in Several Variables}, publisher={Cambridge University Press}, author={Pemantle, Robin and Wilson, Mark C. and Melczer, Stephen}, year={2024}, collection={Cambridge Studies in Advanced Mathematics},
doi = {10.1017/9781108874144}
}

@article{MeWi:2019,
  title={Higher dimensional lattice walks: Connecting combinatorial and analytic behavior},
  author={Melczer, S. and Wilson, M. C.},
  journal={SIAM Journal on Discrete Mathematics},
  volume={33},
  number={4},
  pages={2140--2174},
  year={2019},
  publisher={SIAM},
  doi={10.1137/18M1220856}
}

@article{MiSi:2020,
title = {The asymptotics of reflectable weighted walks in arbitrary dimension},
journal = {Advances in Applied Mathematics},
volume = {118},
pages = {102043},
year = {2020},
issn = {0196-8858},
doi = {10.1016/j.aam.2020.102043},
author = {Marni Mishna and Samuel Simon}
}

@article{Raschel:2012,
  title={Counting walks in a quadrant: a unified approach via boundary value problems},
  author={Raschel, K.},
  journal={Journal of the European Mathematical Society},
  volume={14},
  number={3},
  pages={749--777},
  year={2012},
  doi={10.4171/JEMS/317}
}

@article{SaZi:1994,
	Author = {B. Salvy and P. Zimmermann},
	Journal = {ACM Transactions on Mathematical Software},
	Number = {2},
	Pages = {163-177},
	Title = {{GFUN}: a Maple package for the manipulation of generating and holonomic functions in one variable},
	Volume = {20},
    Year = {1994},
    doi = {10.1145/178365.178368}
}

@phdthesis{Simon:2023,
  author  = "Samuel Lee Krumm Simon",
  title   = "Walks 'n' Blocks: Asymptotic Enumeration of Weighted Reflectable Walks in $A_1^d$ and $A_2$ and Exploration of Balanced Splittable Hadamard Matrices",
  school  = "Simon Fraser University",
  year    = "2023",
  url = {https://summit.sfu.ca/item/36175}
}
\end{document}